\documentclass{article}
\usepackage{graphicx} 
\usepackage{mathrsfs}
\usepackage{tikz-cd}
\usepackage{hyperref}
\usepackage{amsmath}
\usepackage{quiver}
\usepackage{amsthm}
\usepackage{amssymb}
\usepackage{comment}
\usepackage[skip=10pt]{parskip}
\usepackage[left=2cm,right=2cm,top=2cm,bottom=2cm]{geometry}
\usepackage{ytableau}
\usepackage{multicol}
\renewcommand\labelenumi{\upshape(\roman{enumi})}
\renewcommand\theenumi\labelenumi

\usepackage{tikz}
\usepackage{xcolor}
\usepackage{bm}
\newcommand{\cross}[2]{
\begin{scope}[xshift=#1cm,yshift=-#2cm]
\draw[very thick] (0,-0.5) -- (0.45,-0.5);
\draw[very thick] (0.55,-0.5) -- (1,-0.5);
    \draw[very thick] (0.5,0) -- (0.5,-1);
\end{scope}
}
\newcommand{\arcs}[2]{
\begin{scope}[xshift=#1cm,yshift=-#2cm]
    \draw[very thick] (0,-0.5) arc (-90:0:0.5cm);
    \draw[very thick] (0.5,-1) arc (180:90:0.5cm);
\end{scope}
}

\newcommand{\GRIDwitnum}[1]{
\begin{scope}
\foreach \n in {1,...,#1} \node at (\n-0.5,0.5){$\bm{\mathsf{\n}}$};
\foreach \n in {1,...,#1} \node at (-0.5,-\n+0.5){$\bm{\mathsf{\n}}$};
\draw[black!25!white] (0,-#1cm) |- (#1cm,0);
\foreach \n in {0,...,{#1}} \draw[black!25!white] (\n+1,0) -- (\n+1,\n-#1);
\foreach \n in {0,...,{#1}} \draw[black!25!white] (0,-\n-1) -- (#1-\n,-\n-1);
\foreach \n in {1,...,{#1}} \draw[very thick] (\n-1,-#1+\n-0.5) arc (-90:0:0.5cm);
\end{scope}
} 

\newtheorem{theorem}{Theorem}[section]
\newtheorem{lemma}[theorem]{Lemma}

\newtheorem{proposition}[theorem]{Proposition}

\newtheorem{claim}[theorem]{Claim}

\theoremstyle{definition} 
\newtheorem{example}[theorem]{Example}

\newtheorem{remark}[theorem]{Remark}
\newtheorem{definition}[theorem]{Definition}

\newtheorem{fact}[theorem]{Fact}

\title{Forest Polynomials and Pattern Avoidance}
\author{Annie Guo and Dora Woodruff}
\date{}

\begin{document}

\maketitle

\begin{abstract}
    Forest polynomials, recently introduced by Nadeau and Tewari, can be thought of as a quasisymmetric analogue for Schubert polynomials. They have already been shown to exhibit interesting interactions with Schubert polynomials; for example, Schubert polynomials decompose positively into forest polynomials. We further describe this relationship by showing that a Schubert polynomial $\mathfrak{S}_w$ is a forest polynomial exactly when $w$ avoids a set of $6$ patterns. This result adds to the long list of properties of Schubert polynomials that are controlled by pattern avoidance. 
\end{abstract}

\section{Introduction}

Forest polynomials, first introduced by Nadeau and Tewari in their study of the permutahedral variety \cite{forestpolynomials1}, are a quasisymmetric analogue for Schubert polynomials. They form a linear basis for a quasisymmetric variant of the coinvariant algebra \cite{forestpolynomials1}, represent cohomology classes for a quasisymmetric variant of the flag variety \cite{quasisymmetricflag}, and have rich combinatorial properties, many of which parallel the combinatorial properties of Schubert polynomials (for example, see \cite{quasisymmetricdivideddifferences}). 

Forest polynomials and Schubert polynomials also interact with each other in interesting ways. For instance, Nadeau and Tewari showed that Schubert polynomials decompose positively into forest polynomials \cite{forestpolynomials1}, and pipe dream-like models have been developed for forest polynomials \cite{quasisymmetricdivideddifferences}. Given this, it is natural to ask: when do these two classes of polynomials coincide?

One would hope for the answer to this question to be expressible in terms of \textit{pattern avoidance}. Indeed, pattern avoidance controls many combinatorial and algebraic properties of Schubert polynomials. For example, a Schubert polynomial $\mathfrak{S}_w$ is a flagged Schur polynomial if and only if $w$ avoids $2143$ \cite{LascouxSchutzenberger1985}, it is a standard elementary monomial if and only if $w$ avoids $312$ and $1432$ \cite{singleSEMs}, and its coefficients are all $0$ or $1$ if and only if $w$ avoids a list of $12$ patterns \cite{01Schubert}. Much work has also been dedicated to bounding the principal specialization of $\mathfrak{S}_w$ in terms of pattern counts in $w$ (for example, see \cite{132patterns} or \cite{principalspecializations}). 

Along these lines, our main result  is that the intersection between Schubert polynomials and forest polynomials is also controlled by pattern avoidance: 
\vspace{1em}

\begin{theorem}\label{thm:patternavoidance}
    A Schubert polynomial $\mathfrak{S}_w$ is also a forest polynomial if and only if $w$ avoids the following set of patterns: 
    \[\{1432, 2413, 2431, 14523, 32154, 341265\}\]
\end{theorem}

There is a second interpretation of this result worth mentioning. When they introduced forest polynomials in \cite{forestpolynomials1}, Nadeau and Tewari defined $\Omega$-equivalence, an equivalence relation on the set of reduced words of a given permutation. To each equivalence class $\mathcal{C}$, they attach a corresponding forest polynomial $\mathfrak{F}(\mathcal{C})$.
\vspace{1em}

\begin{theorem}
    if $\mathcal{G}_w$ denotes the set of $\Omega$-equivalence classes of the reduced words of $w^{-1}$, then:

\[\mathfrak{S}_w = \sum_{\mathcal{C} \in \mathcal{G}_w} \mathfrak{F}(\mathcal{C})\]

\end{theorem}

Therefore, asking which Schubert polynomials are  forest polynomials is equivalent to asking which permutations $w$ have exactly one $\Omega$-equivalence class. 

\subsection{Organization}
 In section $2$, we will go over the combinatorial model for forest polynomials we use to prove Theorem \ref{thm:patternavoidance}: \textit{labeled binary forests.} (This model was also introduced by Nadeau and Tewari in \cite{forestpolynomials1}). Then, we will quickly review some combinatorics of permutations and pipe dreams. In section $3$, we describe a weight-preserving injection $\psi_w$ sending labelings of a binary forest $F$ to (reduced) pipe dreams for a permutation $w$ (with the \textit{same code} as $F$). In section $4$, we show that if $w$ contains a forbidden pattern, then our injection $\psi_w$ cannot be surjective, proving one direction of Theorem \ref{thm:patternavoidance}. Finally, in section $5$, we show that the forbidden patterns listed in Theorem\ref{thm:patternavoidance} are the only possible obstructions to the bijectivity of $\psi_w$, completing the proof. 

\section{Preliminaries}

\subsection{Forest polynomials and indexed forests}

This subsection roughly summarizes the definitions given in section $3$ of \cite{forestpolynomials1}. Let $S$ be a finite set of positive integers; then, $S$ decomposes as $I_1 \sqcup I_2 \dots \sqcup I_k$, where each $I_i$ is a maximal subset of \textit{consecutive} integers in $S$. A \textit{binary indexed forest $F$ supported on $S$} is the data of a binary tree $T_i$ for each $i \leq j$, whose leaves are attached to the integers $I_i$. Figure \ref{fig:example_forest} provides an example of such a binary indexed forest $F$.

\begin{figure}
\centering
    \begin{tikzpicture}[x=0.8cm,y=0.7cm,baseline=(base)]
  \coordinate (base) at (1,0);
  \draw (1,0) -- (10,0);
  \foreach \i in {1,...,10}{
    \node[below=2pt] at (\i,0) {\small $\i$};
    \draw[thick] (\i-0.1,0.1) -- (\i+0.1,-0.1);
    \draw[thick] (\i-0.1,-0.1) -- (\i+0.1,0.1);
  }

  \fill (2.5,1.5) circle (1.5pt);
  \fill (3.5,0.5) circle (1.5pt);
  \fill (3,1.0)  circle (1.5pt);
  \fill (3.5,2.5) circle (1.5pt);

  \draw[thick]
    (1,0) -- (2.5,1.5) -- (3,1)
    (3,0) -- (3.5,0.5) -- (4,0)
    (2,0) -- (3,1.0) -- (3.5,0.5)
    (2.5,1.5) -- (3.5, 2.5) -- (5.5, 0.5);

  \fill (8.5,0.5) circle (1.5pt);
  \draw[thick] (8,0) -- (8.5,0.5) -- (9,0);

  \fill (5.5,0.5) circle (1.5pt);
  \draw[thick] (5,0) -- (5.5,0.5) -- (6,0);

  \node[black,scale=1] at (5.5,4) {$F$};
\end{tikzpicture}
    \caption{A binary indexed forest $F$.}
    \label{fig:example_forest}
\end{figure}
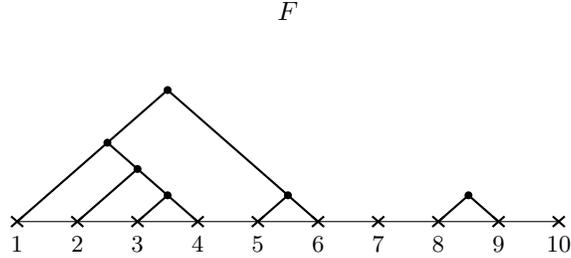

The non-leaf vertices of $F$ are called \textit{interior vertices}, denoted by $\text{IN}(F)$. These binary indexed forests will form our indexing set for the forest polynomials $\mathfrak{F}_F$. 

There is an important function $\rho_F: \text{IN}(F) \to \mathbb{Z}$. It sends an interior vertex $v \in F$ to the leaf label reached by following left edges all the way down. Figure \ref{fig:example_labeled} illustrates the labeling yielded from applying $\rho_F$ on each $v \in F$. 

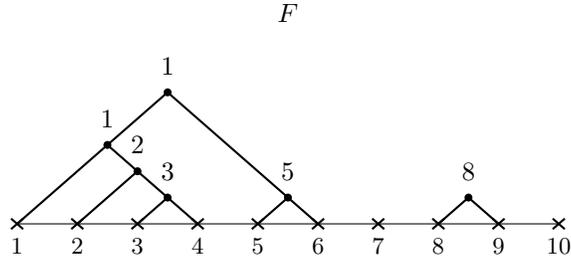
\begin{figure}
    \centering
    \begin{tikzpicture}[x=0.8cm,y=0.7cm,baseline=(base)]
  \coordinate (base) at (1,0);
  \draw (1,0) -- (10,0);
  \foreach \i in {1,...,10}{
    \node[below=2pt] at (\i,0) {\small $\i$};
    \draw[thick] (\i-0.1,0.1) -- (\i+0.1,-0.1);
    \draw[thick] (\i-0.1,-0.1) -- (\i+0.1,0.1);
  }

  \fill (2.5,1.5) circle (1.5pt);
  \fill (3.5,0.5) circle (1.5pt);
  \fill (3,1.0)  circle (1.5pt);
  \fill (3.5,2.5) circle (1.5pt);

  \draw[thick]
    (1,0) -- (2.5,1.5) -- (3,1)
    (3,0) -- (3.5,0.5) -- (4,0)
    (2,0) -- (3,1.0) -- (3.5,0.5)
    (2.5,1.5) -- (3.5, 2.5) -- (5.5, 0.5);

  \node at (2.5,2) {$1$};
  \node at (3,1.5)  {$2$};
  \node at (3.5,1) {$3$};
  \node at (3.5,3) {$1$};

  \fill (8.5,0.5) circle (1.5pt);
  \draw[thick] (8,0) -- (8.5,0.5) -- (9,0);
  \node[text=black] at (8.5,1.0) {$8$};

  \fill (5.5,0.5) circle (1.5pt);
  \draw[thick] (5,0) -- (5.5,0.5) -- (6,0);
  \node[text=black] at (5.5,1.0) {$5$};

  \node[black,scale=1] at (5.5,4) {$F$};
\end{tikzpicture}
    \caption{Forest labeling $f_F$ given by $\rho_F$ for $F$ from Figure \ref{fig:example_forest}}
    \label{fig:example_labeled}
\end{figure}

\begin{definition}\label{def:validlabelings}
    A forest labeling $f_F: \text{IN}(F) \to \mathbb{Z}$ is called \textit{valid} if it satisfies the following three conditions: 
    \begin{enumerate}
        \item For all $v$, $f_F(v) \leq \rho_F(v)$
        \item If $w$ is a left child of $v$, then $f_F(v) \leq f_F(w)$ and
        \item If $w$ is a right child of $v$, then $f_F(v) < f_F(w)$. 
    \end{enumerate}
\end{definition}

Now, we use the following combinatorial model of Nadeau and Tewari as our definition for forest polynomials:
\vspace{2em}

\begin{definition}
    Let $F$ be a binary indexed forest. The forest polynomial $\mathfrak{F}_F$ is defined by 

    \[\sum_{f_F}\prod_{v \in \text{IN}(F)} x_{f_F(v)}\]

    where the sum runs over all valid forest labelings of $F$. 
\end{definition}

\vspace{1em}

\begin{example}
    Consider the example given in Figure \ref{fig:forest_labelings}. We compute the forest polynomial by adding a term for each valid labelings, obtaining:
    \[\mathfrak{F}_F=x_1^2x_2x_3(x_2+x_3+x_4+x_5)(x_1+x_2+\dots+x_8)\]

\begin{figure}
    \centering
    \begin{tikzpicture}[x=0.8cm,y=0.7cm,baseline=(base)]
  \coordinate (base) at (1,0);
  \draw (1,0) -- (10,0);
  \foreach \i in {1,...,10}{
    \node[below=2pt] at (\i,0) {\small $\i$};
    \draw[thick] (\i-0.1,0.1) -- (\i+0.1,-0.1);
    \draw[thick] (\i-0.1,-0.1) -- (\i+0.1,0.1);
  }

  \fill (2.5,1.5) circle (1.5pt);
  \fill (3.5,0.5) circle (1.5pt);
  \fill (3,1.0)  circle (1.5pt);
  \fill (3.5,2.5) circle (1.5pt);

  \draw[thick]
    (1,0) -- (2.5,1.5) -- (3,1)
    (3,0) -- (3.5,0.5) -- (4,0)
    (2,0) -- (3,1.0) -- (3.5,0.5)
    (2.5,1.5) -- (3.5, 2.5) -- (5.5, 0.5);

  \node at (2.5,2) {$1$};
  \node at (3,1.5)  {$2$};
  \node at (3.5,1) {$3$};
  \node at (3.5,3) {$1$};

  \fill (8.5,0.5) circle (1.5pt);
  \draw[thick] (8,0) -- (8.5,0.5) -- (9,0);
  \node[text=black] at (8.5,1.0) {$\{1, 2, \dots8 \}$};

  \fill (5.5,0.5) circle (1.5pt);
  \draw[thick] (5,0) -- (5.5,0.5) -- (6,0);
  \node[text=black] at (6.2,1.0) {$\{2,3,4,5\}$};

  \node[black,scale=1] at (5.5,4) {$F$};
\end{tikzpicture}
    \caption{Options for forest labelings of $f_F$ of forest $F$}
    \label{fig:forest_labelings}
\end{figure}
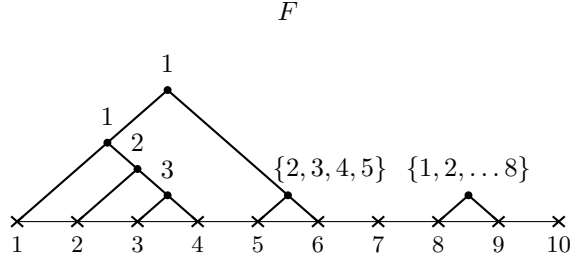
\end{example}

\subsection{Lehmer codes}

\begin{definition}
    The \textit{Lehmer code} $\text{code}(w)$ of a permutation $w \in S_n$ is the vector $(L(1), L(2) \dots L(n))$, where 
    \[L(i) = \#\{j > i| w(j) < w(i)\}\]
\end{definition}

For example, $\text{code}(15342) = (0, 3, 1, 1, 0)$. 

Note that $w \to \text{code}(w)$ defines a bijection between $S_{\mathbb{Z}}$ and the set of all nonnegative integer vectors with finitely many nonzero entries. 
\vspace{1em}

\begin{fact}
    The leading monomial (in the reverse lexicographic ordering) of the Schubert polynomial $\mathfrak{S}_w$ is $\bf{x}^{\text{code}(w)}$.  
\end{fact}

Nadeau and Tewari gave an analogous notion of Lehmer codes for binary indexed forests: 
\vspace{1em}

\begin{definition}\label{def:lehmercodeforests}
    If $F$ is a binary indexed forest, then $\text{code}(F) = (L(1), L(2) \dots L(n))$, where 
    \[L(i) = \#\{v \in \text{IN}(F)| \rho_F(v) = i\}\]
\end{definition}
\vspace{1em}

\begin{example}

For the forest $F$ in Figure \ref{fig:forest_labelings}, we can compute the code as $\text{code}(F)=(2,1,1,0,1,0,0,1,0,0,\dots)$

\end{example}

The map $F \to \text{code}(F)$ also sets up a bijection between binary indexed forests and vectors of nonnegative integers with finite nonzero support. Furthermore, it is simple to see from Definition \ref{def:validlabelings} that the leading monomial (in reverse lexicographic ordering) of $\mathfrak{F}_F$ is also given by $\text{code}(F)$ (which proves that $\{\mathfrak{F}_F\}_F$ forms a basis for $\mathbb{Z}[x_1, x_2 \dots]$). 
\vspace{1em}

\begin{remark}\label{rmk:easyrmk}
    A simple observation follows: if $\mathfrak{S}_w$ is indeed equal to a forest polynomial $\mathfrak{F}_F$, there is only one possible candidate for which forest polynomial it could be: $F$ must be the unique forest satisfying $\text{code}(T) = \text{code}(w)$, since the leading monomials in the reverse lexicographic order must match. Thus, we only need to compare $\mathfrak{S}_w$ to one candidate forest polynomial. 
\end{remark}

\subsection{Pipe dreams and ladder moves}
\textit{Pipe dreams} are an important combinatorial model for Schubert polynomials developed by Billey-Jockusch-Stanley \cite{pipedreams1} (see also \cite{pipedreams2} and \cite{pipedreams} for more in-depth background on pipe dreams). Formally, a (reduced) pipe dream of $w$ is a finite subset $D \subset \mathbb{Z}_{>0} \times \mathbb{Z}_{>0}$ satisfying the following condition: the ordered product of the generators $s_{i+j-1}$ over all $(i,j)\in D$ (taken row by row from left to right, bottom to top) is a reduced word for $w$. From now on, all pipe dreams we consider will be reduced. 

 The set of pipe dreams of $w$ is denoted $\text{PD}(w)$. Every pipe dream $D$ of $w$ can be assigned a \textit{weight}: 
\[\text{wt}(D) = \prod_{(i,j) \in D} x^i\]

Bergeron and Billey \cite{pipedreams} proved that pipe dreams for $w$ correspond to monomials in $\mathfrak{S}_w$:

\vspace{2em}

\begin{theorem}\label{thm:mainpipedreamthm}
    \[\mathfrak{S}_w = \sum_{D \in \text{PD}(w)} \text{wt}(D)\]
\end{theorem}

Pictorially, a pipe dream $D$ can be viewed as a \textit{strand diagram}: place a \textit{crossing} at every $(i,j) \in D$ and a pair of \textit{elbows} at every $(i, j) \notin D$. The \textit{row} of a crossing is $j$, and its \textit{column} is $i$. This process results in a pseudo-line arrangement, where each strands connects $(k, 0)$ to $(0, w(k))$ and no two strands intersect more than once. 

\vspace{2em}

\begin{example}\label{example:4132}

Below are the two pipe dreams for $w = 4132$. The weight of the first pipe dream is $x_1^3x_3$ and the weight of the second is $x_1^3x_2$, so, by Theorem \ref{thm:mainpipedreamthm}, $\mathfrak{S}_{4132} = x_1^3x_3+ x_1^3x_2$. 

\begin{center}\includegraphics[scale=0.7]{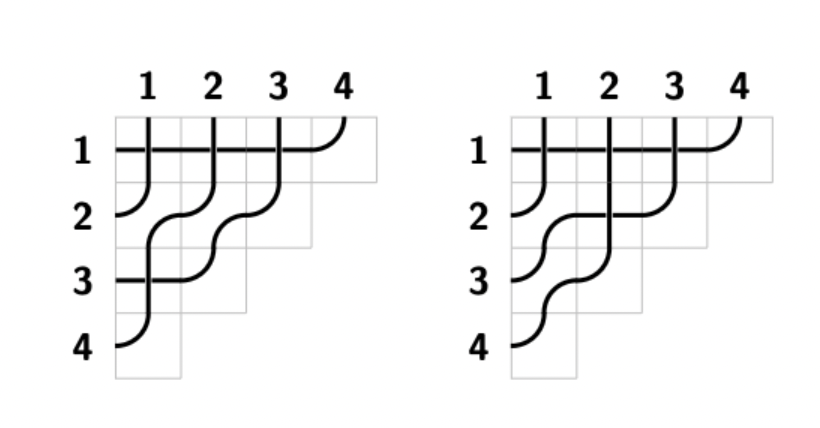}
\end{center}
\end{example}

Every permutation has a distinguished pipe dream, called the \textit{bottom pipe dream}, also defined and studied in \cite{pipedreams}:
\vspace{1em}

\begin{definition}\label{def:lehmer}
    The \emph{bottom pipe dream} of $w$ is the unique left-justified pipe dream for $w$ where the number of crossings in row $i$ is $L_i$ (where $L(w)$ is the Lehmer code of $w$). 
\end{definition}

For example, the bottom pipe dream of $4132$ is given by the left pipe dream in Example \ref{example:4132}.

Bergeron-Billey \cite{pipedreams} also gave a simple way to generate every pipe dream of $w$ given its bottom pipe dream.
\vspace{1em}

\begin{definition}
    A \textit{ladder move of order $k$} of $D \in \text{RC}(w)$ applied at a crossing $(i,j) \in D$ produces a new $D' \in \text{RC}(w)$ given by $D \setminus \{(i,j)\} \cup \{(i-k-1, j+1)\}$. A ladder move can only be applied at $(i,j) \in D$ if two conditions are satisfied:
    \begin{enumerate}
        \item $(i,j) \in D$, and $(i, j+1), (i-k-1, j), (i-k-1, j+1) \notin D$ and 
        \item For all $i-k \leq i' < i$, $(i', j), (i', j+1) \in D$
    \end{enumerate}
\end{definition}

Informally, the crossing $(i,j)$ `climbs a ladder' made of other crossings of $D$ and arrives at a new square $(i-k-1, j+1)$. 

\vspace{2em}

\begin{example}

A ladder move of order $2$ carries the bottom left crossing up by three rows and right by one column: 

    \begin{center}
        \includegraphics[scale=0.6]{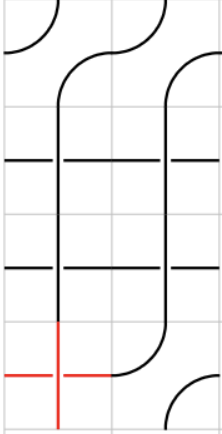}
     \hspace{3em} \includegraphics[scale=0.6]{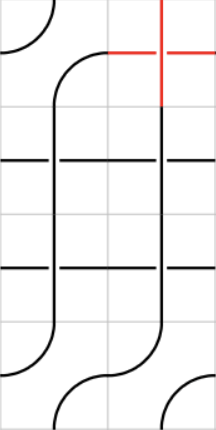}
     \end{center}
\end{example}

\vspace{2em}

\begin{theorem}[{\cite{pipedreams}}]\label{thm:laddermove}
    Every reduced pipe dream for $w$ can be obtained from the bottom pipe dream for $w$ by a sequence of ladder moves. 
\end{theorem}

A ladder move of order $0$ is also called a \textit{simple ladder move}, only moving a crossing up one square and right one square. For example, the pipe dream on the right in Example \ref{example:4132} is obtained from the bottom pipe dream of $4132$ by applying one simple ladder move. 

Finally, it is useful to define the \textit{northeast diagonal} of a crossing, which consists of all grid squares lying on a ray of slope $1$ emitting northeast from that crossing. So, simple ladder moves just slide crossings along their northeast diagonals.

We will often label these northeast diagonals $d_i$ from top to bottom, as depicted below: 

\begin{center}
\begin{ytableau}
    $1$ & $2$ & $3$ & $4$\\
    $2$ & $3$ & $4$ & $5$\\
    $3$ & $4$ & $5$ & $6$\\
    $4$ & $5$ & $6$ & $7$\\
\end{ytableau}
\end{center}

Similarly, we will generally read crossings from left to right along rows (so that the `first' crossing in a row is the leftmost one). 

\subsection{Pattern avoidance}

Let $w \in S_n$ be a permutation. We say a permutation $W \in S_N$, where $N$ is usually much larger than $n$, \textit{contains the pattern $w$} if there exist indices $i_1 < i_2 \dots < i_n$ such that $W_{i_1}, W_{i_2} \dots W_{i_n}$ are in the same relative order as $w_1, w_2 \dots w_n$. 

One of the patterns listed in Theorem \ref{thm:patternavoidance} is $1432$, and its role in the proof is different from the other patterns listed. It arises because of the following important theorem of Gao \cite{principalspecializations}:
\vspace{1em}

\begin{theorem}\label{thm:1432simple}
    A permutation $w$ avoids the pattern $1432$ if and only if every reduced pipe dream for $w$ can be obtained from $P_{\text{bott}}(w)$ via a sequence of \textit{simple} ladder moves.
\end{theorem}

That is, any pipe dream for $w$ can be obtained by sliding crossings along their starting diagonals.

With these preliminaries out of the way, we may now begin our proof of Theorem \ref{thm:patternavoidance}. 

\section{An injection from labelings to pipe dreams}

\subsection{Overarching proof strategy}

In order to prove Theorem \ref{thm:patternavoidance}, our strategy is as follows. Let $F$ and $w$ be a binary forest and a permutation respectively, with $\text{code}(F) = \text{code}(w)$. First, we will describe a weight-preserving injection $\psi_w$ from valid labelings of $F$ to (reduced) pipe dreams of $w$. 

Then, to show one direction of Theorem \ref{thm:patternavoidance}, we essentially show that if $\psi_w$ is not a surjection for a permutation $w$, then $\psi_{w'}$ is not a surjection for any permutation $w'$ containing $w$ as a pattern (Lemma \ref{lemma:patterncontainment}). Given this pattern containment fact, it suffices to manually check that the forbidden patterns listed do not yield forest polynomials. 

Finally, we prove the opposite direction: that the patterns appearing in Theorem \ref{thm:patternavoidance} are the only possible obstructions to the bijectivity of $\psi_w$. This direction is somewhat more onerous and needs more casework. 

\subsection{The injection $\psi_w$}

Throughout the rest of this section, we will let $F$ and $w$ be a binary forest and a permutation \textit{with the same code} (following Remark \ref{rmk:easyrmk}). 
\vspace{1em}

\begin{definition}[Corresponding crossings and vertices]\label{def:correspondence}

Mark the crossings of row $i$ of $P_{\text{bott}}(w)$ with $1, 2 \dots L_w(i)$ from left to right, and mark the (interior) vertices in left branch $i$ of $F$ with $1, 2 \dots L_F(i)$ from bottom to top. We say that crossing $j$ in row $i$ of $P_{\text{bott}}(w)$ \textit{corresponds} to vertex marked $j$ in left branch marked $i$ of $F$. 
\end{definition}

If $v \in T$ is an interior vertex, we will use the notation $c_v$ to denote the crossing in $P_{\text{bott}}(w)$ corresponding to $v$. 
\vspace{1em}

\begin{example}
    Let $w = 4132$. Then, $P_{\text{bott}}(w)$ is:


\begin{center}
    \begin{tikzpicture}[scale=0.6]
    \GRIDwitnum{4}
    \textcolor{red}{\cross{0}{0}}
    \textcolor{orange}{\cross{1}{0}}
    \textcolor{violet}{\cross{2}{0}}

    \arcs{0}{1}
    \arcs{1}{1}
    \textcolor{cyan}{\cross{0}{2}}
\end{tikzpicture}
\end{center}

  In order, let us label the crossings as $c_{v_1}, c_{v_2}, c_{v_3}$ (in the first row), and $c_{v_4}$ (in the third row). These correspond to vertices $v_1, v_2, v_3, v_4$ as labeled in the forest.

\begin{center}
\begin{tikzpicture}[x=0.8cm,y=0.7cm,baseline=(base)]
  \coordinate (base) at (1,0);
  \draw (1,0) -- (5,0);
  \foreach \i in {1,...,5}{
    \node[below=2pt] at (\i,0) {\small $\i$};
    \draw[thick] (\i-0.1,0.1) -- (\i+0.1,-0.1);
    \draw[thick] (\i-0.1,-0.1) -- (\i+0.1,0.1);
  }

  \draw[thick]
    (1,0) -- (1.5,0.5) -- (2,0)
    (1.5,0.5) -- (2.5,1.5) -- (3.5,0.5)
    (3,0) -- (3.5,0.5) -- (4,0)
    (2.5,1.5) -- (3, 2) -- (5, 0);

  \fill[orange] (2.5,1.5) circle (1.5pt);
  \fill[cyan] (3.5,0.5) circle (1.5pt);
  \fill[red] (1.5,0.5)  circle (1.5pt);
  \fill[violet] (3,2) circle (1.5pt);

  \node at (2.5,2) {$v_2$};
  \node at (1.5,1)  {$v_1$};
  \node at (3.5,1) {$v_4$};
  \node at (3,2.5) {$v_3$};

\end{tikzpicture}   
\end{center}

\end{example}

Using definition \ref{def:correspondence}, we can describe a map from valid labelings of $F$ to pipe dreams of $w$: 
\vspace{1em}

\begin{definition}\label{def:injectionfromlabelings}
    Let $f_F$ be a valid labeling of $F$. Then, $\psi_w(f_F)$ is the pipe dream for $w$ obtained from $P_{\text{bott}}(w)$ by sliding $c_v$ up to row $f_F(v)$ via \textit{simple ladder moves} for each $v \in \text{IN}(F)$. 
\end{definition}

We will prove shortly that $\psi_w(f_F)$ can always be defined for any $w$. 

\vspace{1em}
\begin{example}
    Let $f_F$ be the following labeling of $F$:

\begin{center}
\begin{tikzpicture}[x=0.8cm,y=0.7cm,baseline=(base)]
  \coordinate (base) at (1,0);
  \draw (1,0) -- (5,0);
  \foreach \i in {1,...,5}{
    \node[below=2pt] at (\i,0) {\small $\i$};
    \draw[thick] (\i-0.1,0.1) -- (\i+0.1,-0.1);
    \draw[thick] (\i-0.1,-0.1) -- (\i+0.1,0.1);
  }

  \draw[thick]
    (1,0) -- (1.5,0.5) -- (2,0)
    (1.5,0.5) -- (2.5,1.5) -- (3.5,0.5)
    (3,0) -- (3.5,0.5) -- (4,0)
    (2.5,1.5) -- (3, 2) -- (5, 0);

  \fill[orange] (2.5,1.5) circle (1.5pt);
  \fill[cyan] (3.5,0.5) circle (1.5pt);
  \fill[red] (1.5,0.5)  circle (1.5pt);
  \fill[violet] (3,2) circle (1.5pt);

  \node at (2.5,2) {$1$};
  \node at (1.5,1)  {$1$};
  \node at (3.5,1) {$2$};
  \node at (3,2.5) {$1$};

\end{tikzpicture} 
\end{center}


Which pipe dream for $4132$ should it correspond to? The crossings $c_{v_1}, c_{v_2}, c_{v_3}$ (red, orange, and purple) in row $1$ all remain in row $1$, but the crossing $c_{v_4}$ (light blue) in row $3$ should slide up to row $2$ via a simple ladder move:
\begin{center}

    \begin{tikzpicture}[scale=0.6]
    \GRIDwitnum{4}
    \textcolor{red}{\cross{0}{0}}
    \textcolor{orange}{\cross{1}{0}}
    \textcolor{violet}{\cross{2}{0}}

    \arcs{0}{1}
    \textcolor{cyan}{\cross{1}{1}}
    \arcs{0}{2}
\end{tikzpicture}
\end{center}

\end{example}

\begin{proposition}\label{prop:psiiswell-defined}
    The map $\psi_w$ can always be defined for any $w$. It is an injective map
    \[\{\text{valid labelings of } F\} \to \{\text{reduced pipe dreams for } w\}\]

    Furthermore, every pipe dream in its image can be obtained from $P_{\text{bott}}(w)$ via a sequence of simple ladder moves. 
\end{proposition}

The fact that $\psi_w$ is injective is clear by construction; the claim that $\psi_w$ can always be defined takes more work. To prove it, we will need a couple more facts. 

\subsection{Adjacencies between crossings}

It will be important to be able to read off, from $P_{\text{bott}}(w)$, when two crossings correspond (under definition \ref{def:injectionfromlabelings}) to a pair of \textit{adjacent} vertices in $F$. 
\vspace{1em}

\begin{remark}
    It is easy to describe when two crossings $c_1, c_2$ in $P_{\text{bott}}(w)$ correspond to a pair $(v,w)$ such that $w$ is a \textit{left} child of $v$. This occurs exactly when $c_2$ is in the same row as $c_1$, and one space to the right of $c_1$. What is more subtle is describing when $w$ is a \textit{right} child of $v$. This relationship is captured in the following definition.  
\end{remark}
\vspace{1em}

\begin{definition}[Covering crossings]\label{definition:matching}
   We recursively define a certain relation between crossings of $P_{\text{bott}}(w)$ as follows. First, consider row $1$. Find the first nonempty row $j$ below row $1$; then, the $(j-1)$st crossing of row $1$ \textit{covers} the rightmost crossing of row $j$. (The first $(j-2)$ crossings of row $1$ do not cover anything). We say that the first $j-1$ crossings of row $1$ are all `completed.'  

   Recursively, apply the same procedure to row $j$. Only return to row $1$ when all crossings in row $j$ have been `completed.' 
   
    Then, when all crossings in row $1$ have been completed, find the next nonempty row $i$ with non-completed crossings, and start from the beginning, with $i$ playing the same role as $1$. If there is no such row $i$, we are done. 
\end{definition}
\vspace{1em}

\begin{example}
    Let's consider the bottom pipe dream for $w = 41532$, for which $L(w) = (3, 0, 2, 1, 0)$: 

    \begin{center}
            \begin{tikzpicture}[scale=0.6]
    \GRIDwitnum{6}
   \cross{0}{0}
   \textcolor{red}{\cross{1}{0}}
   \cross{2}{0}
   \arcs{3}{0}
   \arcs{4}{0}

    \arcs{0}{1}
    \arcs{1}{1}
    \arcs{2}{1}
    \arcs{3}{1}
    \textcolor{red}{\cross{0}{2}}
    \textcolor{cyan}{\cross{1}{2}}
    \arcs{2}{2}
    \textcolor{cyan}{\cross{0}{3}}
    \arcs{1}{3}
    \arcs{0}{4}
\end{tikzpicture}
\end{center}

The first crossing in row $1$ does not cover anything, because row $2$ is empty. The second crossing in row $1$ covers the second crossing in row $3$, so now we move to row $3$. The first crossing in row $3$ covers the first crossing in row $4$. The last crossing in row $1$ does not cover anything (there are no more crossings below left to cover) so we are done. 

So, there is a covering relationship between the red crossings and between the blue crossings; that is, crossing $2$ of row $1$ covers crossing $2$ of row $3$, crossing $1$ of row $3$ covers crossing $1$ of row $4$, and those are the only coverings. We encourage the reader to draw the corresponding binary forest and verify that Lemma \ref{lemma:crossingmatching} holds in this case. 
\vspace{1em}

\end{example}
The following fact becomes clear enough, following directly from Definition \ref{def:correspondence}, once one considers several examples:
\vspace{1em}
\begin{lemma}\label{lemma:crossingmatching}
    Crossing $c_1$ covers crossing $c_2$, according to the definition above, if and only if $c_2$ is the right child of $c_1$. 
\end{lemma}

Using Lemma \ref{lemma:crossingmatching}, we may prove Proposition \ref{prop:psiiswell-defined}:

\begin{proof}
We want to show that $\psi_w(f_F)$ can be defined for any valid labeling $f_F$ of $F$.  
    That is, if $f_F$ is a valid labeling for $F$, it must be possible possible, for each interior vertex $v$ of $F$, to slide the corresponding crossing $c_v \in P_{\text{bott}}(w)$ up to row $f_F(v)$ using only simple ladder moves. Informally, as $c_v$ slides up and to the right along its diagonal, its path should never be obstructed from making more simple ladder moves by another crossing. 

    Reading crossings of $P_{\text{bott}}(w)$ from right to left and top to bottom, we try to slide each crossing $c_v$ up to row $f_F(v)$ in that order. For contradiction, consider the first time at which a crossing $c_v$ cannot slide up into its desired row. There are two possible reasons why $c_v$ might be obstructed from making one more simple ladder move: there could be another crossing directly to its right, or another crossing directly above it (or directly above it and one square to its right). For example, in each diagram below, the bottom left crossing `wants' to slide up to the higher row, but is obstructed either by a crossing to its right, or a crossing directly above it.

\begin{centering}
    \includegraphics[scale=1.5]{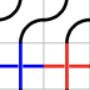}
    \hspace{8mm} 
    \includegraphics[scale=1.5]{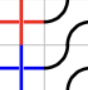}

\end{centering}

\textbf{Case 1:}
    
    Assume that the first obstruction occurs: $c_v$ is blocked from taking another step by a $c_w$ directly to its right. Then, it is straightforward to check that $c_w$ and $c_v$ must have started out as neighbors in the same row. So, according to Definition \ref{def:correspondence}, $w$ is a left child of $v$ in $F$. But then, $f_F(w) \geq f_F(v)$ by the definition of valid labelings. Since $c_w$ is already assumed to be in its desired row $f_F(w)$, $c_v$ is also already in its desired row $f_F(v)$. 

\textbf{Case 2:}

    Now, we consider what happens if there is a second crossing $c_w$ directly above (or directly above and to the right of) $c_v$ that obstructs it from making another simple ladder move. 
    
    \begin{claim} The vertex $v$ is in the right subtree of a descendant of the vertex $w$. 
    \end{claim}

    Consider the bottom pipe dream $P_{\text{bott}}(w)$. In $P_{\text{bott}}$, say the row containing $c_w$ has $l$ crossings. Consider the $l$ rows beneath this row. It follows from the recursive definition of covering crossings (Definition \ref{definition:matching}) and Lemma \ref{lemma:crossingmatching} that every crossing these $l$ rows will be covered either by $c_w$, or one of its descendants in its \textit{right} subtree. 

    If $c_v$ ever slides directly underneath $c_w$ (after applying some sequence of simple ladder moves to both crossings), then the diagonal of $c_v$ is one lower than the diagonal of $c_w$. (Similarly, if $c_w$ is one square above and one square to the right of $c_v$ after some sequence of simple ladder moves, then $c_w$ and $c_v$ share a diagonal). In both of these cases, in $P_{\text{bott}}(w)$, $c_v$ must be in one of the $l$ rows beneath $c_w$.  
    So, $v$ must in the right subtree of $w$. Now, since $v$ is a right descendant of $w$, we must have $f_F(v) > f_F(w)$. But since $c_w$ was already assumed to be slid into row $f_F(w)$, we do not need to slide $c_v$ any higher. 

    To quickly summarize the argument, if a crossing $c_v$ is obstructed by any other crossing $c_w$ as it slides up to row $f_F(v)$, then the obstructing crossing $c_w$ must satisfy $f_F(w) \leq f_F(v)$ (with strict inequality if $c_w$ is above $c_v$). But this means that we have already slid $c_v$ up to $f_F(v)$.
\end{proof}

\begin{remark}\label{rmk:diagonalremark}
Suppose we have crossings $c_w,c_v$ corresponding to vertices $w,v$. Then if $c_w$ lies in the diagonal $d_v-1$ or greater in $P_{\text{bott}}(w)$ (equivalently $c_v$ lies in the diagonal $d_w+1$ or smaller), $v$ is in the subtree of $w$. 
\end{remark}
\begin{proof}
If $c_w$ has only empty rows following it, then the subtree of $w$ is comprised of all rows which intersect diagonal $d_w+1$. Nonempty rows can only increase the number of rows which are contained in the subtree of $w$, so any such $v$ as described in the remark must be in the subtree of $w$.
\end{proof}

\section{Necessary direction}
In this section, we show that if a permutation $w$ contains any of the permutations $\{1432, 2413, 2431, 14523, 32154, 341265\}$, then its corresponding Schubert polynomial $\mathfrak{S}_w$ is not a forest polynomial.

\subsection{The forbidden patterns themselves}

Although it is easy to show by hand that $\mathfrak{S}_w$ is not a forest polynomial for each forbidden $w$, it is useful to understand what is `bad' about these patterns for our proof of the necessary direction. Below are the bottom pipe dreams of the six forbidden patterns:

\begin{multicols}{6}
\begin{center}
    \[\begin{tikzpicture}[scale=0.4]
    \GRIDwitnum{4}
    \arcs{0}{0}
    \arcs{1}{0}
    \arcs{2}{0}
    \cross{0}{1}
    \cross{1}{1}
    \cross{0}{2}
\end{tikzpicture}
    \]
    $w=1432$
\end{center}
    \columnbreak
\begin{center}
    \[
    \begin{tikzpicture}[scale=0.4]
    \GRIDwitnum{4}
    \textcolor{red}{\cross{0}{0}}
    \arcs{1}{0}
    \arcs{2}{0}
    \cross{0}{1}
    \textcolor{blue}{\cross{1}{1}}
    \arcs{0}{2}
\end{tikzpicture}
    \]
    $w=2413$ \\
    $c_w$: row 1, crossing 1 \\ 
    $c_v$: row 2, crossing 2 \\ 
\end{center}
    \columnbreak

\begin{center}
    \[\begin{tikzpicture}[scale=0.4]
    \GRIDwitnum{4}
    \textcolor{red}{\cross{0}{0}}
    \arcs{1}{0}
    \arcs{2}{0}
    \cross{0}{1}
    \textcolor{blue}{\cross{1}{1}}
    \cross{0}{2}
\end{tikzpicture}
    \]
    $w=2431$ \\
    $c_w$: row 1, crossing 1 \\ 
    $c_v$: row 2, crossing 2 \\ 
\end{center}
    \columnbreak
\begin{center}
    \[\begin{tikzpicture}[scale=0.4]
    \GRIDwitnum{5}
    \arcs{0}{0}
    \arcs{1}{0}
    \arcs{2}{0}
    \arcs{3}{0}
    \textcolor{red}{\cross{0}{1}}
    \cross{1}{1}
    \arcs{2}{1}
    \cross{0}{2}
    \textcolor{blue}{\cross{1}{2}}
    \arcs{0}{3}
\end{tikzpicture}
    \]
    $w=14523$ \\
    $c_w$: row 2, crossing 1 \\ 
    $c_v$: row 3, crossing 2 \\ 
\end{center}
    
    \columnbreak
\begin{center}
    \[\begin{tikzpicture}[scale=0.4]
    \GRIDwitnum{5}
    \cross{0}{0}
    \textcolor{red}{\cross{1}{0}}
    \arcs{2}{0}
    \arcs{3}{0}
    \cross{0}{1}
    \arcs{1}{1}
    \arcs{2}{1}
    \arcs{0}{2}
    \arcs{1}{2}
    \textcolor{blue}{\cross{0}{3}}
\end{tikzpicture}
    \]
    $w=32154$ \\
    $c_w$: row 1, crossing 2 \\ 
    $c_v$: row 4, crossing 1 \\ 
\end{center}
    \columnbreak
\begin{center}
    \[\begin{tikzpicture}[scale=0.4]
    \GRIDwitnum{6}
    \cross{0}{0}
    \textcolor{red}{\cross{1}{0}}
    \arcs{2}{0}
    \arcs{3}{0}
    \arcs{4}{0}
    \cross{0}{1}
    \cross{1}{1}
    \arcs{2}{1}
    \arcs{3}{1}
    \arcs{0}{2}
    \arcs{1}{2}
    \arcs{2}{2}
    \arcs{0}{3}
    \arcs{1}{3}
    \textcolor{blue}{\cross{0}{4}}
\end{tikzpicture}
    \]
    $w=341265$ \\
    $c_w$: row 1, crossing 2 \\ 
    $c_v$: row 5, crossing 1 \\ 
\end{center}

\end{multicols}

In each of these pipe dreams (except for $1432$), we highlight a red crossing $c_w$ and a blue crossing $c_v$. The pair $(c_w, c_v)$ always satisfies two properties: 

\begin{enumerate}
    \item The corresponding vertex $v$ is the \textit{right} child of $w$ and 
    \item It is possible to slide $c_v$, following some sequence of simple ladder moves, upwards so that $c_v$ and $c_w$ are in the same row. 
\end{enumerate}

These two observations explain why each of these patterns is forbidden: in the forest $F$, we must have $f_F(v) > f_F(w)$, but there is a pipe dream for $w$ such that $\text{row}(v) \leq \text{row}(w)$. That is, our injection $\psi_w$ cannot be a surjection. Our general strategy in this direction will be to prove the existence of such a `bad' parent-right child pair in $P_{\text{bott}}(w)$:
\vspace{1em}

\begin{definition}
    A pair of crossings $(c_w, c_v) \in P_{\text{bott}}(w)$ is called a bad parent-child pair if:
    
    \begin{enumerate}
        \item $c_w$ corresponds to the right child of $c_v$, but
        \item After some sequence of simple ladder moves (applied to $c_w$ or other crossings), we can obtain a pipe dream in which $c_w$ is in the same row as $c_v$. 
    \end{enumerate}
\end{definition}

As explained above, if there exists a bad parent-child pair in $P_{\text{bott}}(w)$, then $\mathfrak{S}_w$ cannot be a forest polynomial. In fact, using Proposition \ref{prop:psiiswell-defined}, we have:
\vspace{1em}

\begin{remark}\label{rmk:usefulremark}
    Suppose $w$ avoids $1432$. Then, $\mathfrak{S}_w \neq \mathfrak{F}_F$ if and only if there exists a `bad' parent-right child pair as above. 
\end{remark}

This remark will also be helpful for the sufficient direction. 

\subsection{Pattern avoidance and $\psi_w$}

First, we treat the pattern $1432$ separately:
\vspace{1em}

\begin{lemma}
    If $w$ contains $\pi = 1432$, $\mathfrak{S}_w$ is not a forest polynomial.
\end{lemma}

\begin{proof}
    According to Proposition \ref{prop:psiiswell-defined}, the image of the injection $\psi$ contains only pipe dreams which can be obtained from $P_{\text{bott}}(w)$ via simple ladder moves. However, $w$ avoids $1432$ if and only if every reduced pipe dream for $w$ can be obtained from $P_{\text{bott}}(w)$ via a sequence of simple ladder moves. Therefore, if $w$ contains $1432$, the monomials appearing $\mathfrak{S}_w$ are a strict superset of the monomials appearing in $\mathfrak{F}_T$. 
\end{proof}

Now, let us consider only permutations which avoid $1432$. Our goal is to prove that if $\mathfrak{S}_{\pi}$ is not a forest polynomial for a pattern $\pi$ avoiding $1432$, then neither is $\mathfrak{S}_{w'}$, for any $w'$ containing $w$ as a pattern.
\vspace{1em}

\begin{definition}\label{def:insertion}
    Given some permutation $w \in S_n$, let us define \textit{insertion} at index $i$ of value $k$ as $w' = I_i^k(w)$, where we obtain $w' \in S_{n+1}$ by adding at index $i$ the value $k\in [n+1]$. Then, we increase by one all values $w(j)$ with $w(j) \ge k$. 
\end{definition}

For example, $I_2^4(1342) = 12453$. The following fact is standard:
\vspace{1em}

\begin{lemma}\label{lemma:insertion}
    Given a permutation $w$ containing a pattern $\pi$, we can obtain $w$ from $\pi$ by repeatedly performing insertions.
\end{lemma}

    Inserting values has a simple effect on $\text{code}(w)$. Namely, the bottom pipe dream of $I_i^k(w)$ is obtained from $P_{\text{bott}}(w)$ by inserting a new (possibly empty) row of crossings right before row $i$, and then, for all $j < i$, adding one more crossing to row $j$ if $w(j) > w(i)$. For example, $L(1342) = (0,1,1,0)$, and $L(12453) = (0,0, 1, 1, 0)$. 

\vspace{1em}
\begin{lemma}\label{lemma:patterncontainment}
    If a permutation $\pi$ satisfies $\mathfrak{S}_\pi \neq \mathfrak{F}_T$, then any permutation $w$ which contains $\pi$ satisfies $\mathfrak{S}_w \neq \mathfrak{F}_T$.
\end{lemma}
\begin{proof}

As remark \ref{rmk:usefulremark}, we can find a `bad' parent-right child pair $c_w, c_v$ in $P_{\text{bott}}(\pi)$ (where $c_v$ can slide up to the same row as $c_w$).

    By Lemma \ref{lemma:insertion}, we can repeatedly perform insertions to eventually obtain $w$ from $\pi$. We will show that if we begin with such a pattern $\pi$ and a pair of crossings $(c_w,c_v)$, then any insertion $I_i^k(w)$ will preserve the existence of such a pair, and thus we conclude that such a bad pair will exist in $w$.
    
Let $c_w$ slide along the $d_w$th northeast diagonal (northeast diagonals in the grid are labeled from top to bottom), and let $c_v$ slide along the $d_v$th diagonal. Then, we must have $d_v-d_w>1$ (otherwise $c_w$ and $c_v$ would not be a bad pair; $c_w$ would `block' $c_v$ from sliding past it). After each insertion, the parent of $c_v$ is an element $c_k$ on some diagonal $d_k \le d_v$ and row $r_k \ge r_v$. 
    \begin{itemize}
        \item If $i\le r_w$ (insert before row $r_w$) then we retain the same pair $(c_w,c_v)$.
        \item If $i > r_v$ (insert after row $r_v$), since $\pi(r_v)<\pi(r_w)$,  if $r_w$ obtains a new crossing in $P_{\text{bott}}(w)$, so must $r_v$. If $r_v$ increases then let $v'$ be the crossing to the right of $v$; our new bad pair will be $(w,v')$. 
        
        \item If $r_w<i \le r_v$ (insert rows between $r_w$ and $r_v$), let us take in row $i$ the rightmost crossing, $c_u$ (with corresponding northeast diagonal numbered $d_u$). Assume that $d_u<d_v-1$; then, $c_v$ can make the same simple ladder moves up to row $r_w$ still. 
    \end{itemize}

The only interesting case is when the newly added row of crossings, row $r_u$, is long enough to block $c_v$ from traversing its diagonal up to row $r_w$. 

But if such a row is inserted, then it must be large enough that the rightmost element in row $r_u$, which we denote $c_u$, is on diagonal $d_v-1$ or higher. But then in this case, the desired bad pair can be taken to be $(c_w, c_u)$. 

\end{proof}
Since we already enumerated six permutations which have a bad parent-child pair $(c_w,c_v)$, any permutations containing them will also satisfy $\mathfrak{S}_w \neq \mathfrak{F}_T$. Thus we conclude the necessary direction of Theorem \ref{thm:patternavoidance}.

\section{Sufficient direction}

Now, we show that if a Schubert polynomial $\mathfrak{S}_w$ is not a forest polynomial, then $w$ must contain one of the permutations $\{1432, 2413, 2431, 14523, 32154, 341265\}$. We may assume throughout that $w$ does not contain $1432$, again making use of Theorem \ref{thm:1432simple}. 

As suggested by Remark \ref{rmk:usefulremark}, if $\mathfrak{S}_w \neq \mathfrak{F}_T$ then there must exist a bad parent-right child pair of crossings $(c_w, c_v)$ in $P_{\text{bott}}(w)$. Our goal is to use such a pair of crossings to construct a forbidden pattern in $w$. 

Let us denote the starting rows of $c_w, c_v$ with $r_w, r_v$, respectively, and their corresponding northeast diagonals with $d_w, d_v$. Then we will consider the two following cases:
\begin{enumerate}
    \item Crossing $c_v$ can be slid to the row of $c_w$ without any other crossings obstructing its path. More specifically, in $P_{\text{bott}}(w)$, there are no crossings on diagonal $d_v$ or the diagonal above, $d_v-1$, in the rows between $r_w$ and $r_v$. In this case, we will show that $w$ continues one of the following patterns: $\boxed{2413, 2431, 32154,$ or $341265.}$
    \item Crossing $c_v$ can be slid to the row of $c_w$ only after some other crossing moves; that is, in $P_{\text{bott}}(w)$, there exists some other crossing on diagonal $d_v$ or $d_v-1$, obstructing the movement of $c_v$. In this case, we will show that $w$ contains either \boxed{$1432$  \text{ or } $14523$.}
\end{enumerate}
\vspace{1em}

\begin{example}
    We give a brief example of each of the cases above. First, let $w = 24513$. Then, $P_{\text{bott}}(w)$ is:  

    \begin{center}
    \begin{tikzpicture}[scale=0.6]
    \GRIDwitnum{5}
    \cross{0}{0}
    \arcs{1}{0}
    \arcs{2}{0}
    \arcs{3}{0}

    \cross{0}{1}
    \cross{1}{1}
    \arcs{2}{1}
    
    \cross{0}{2}
    \cross{1}{2}

    \arcs{0}{3}
\end{tikzpicture}
\end{center}

This is an example of the first case. If $c_w$ is the crossing in row $1$ and $c_v$ is the rightmost crossing of row $2$, then $(c_w, c_v)$ form a bad parent-child pair, since $c_v$ can slide into row $r_w$ in one ladder move. Indeed, $w$ contains a $2413$ pattern. 

Now, suppose that $w = 146235$. Then, $P_{\text{bott}}(w)$ is: 

\begin{center}
    \begin{tikzpicture}[scale=0.6]
    \GRIDwitnum{6}
    \arcs{0}{0}
    \arcs{1}{0}
    \arcs{2}{0}
    \arcs{3}{0}
    \arcs{4}{0}

    \cross{0}{1}
    \cross{1}{1}
    \arcs{2}{1}
    \arcs{3}{1}
    
    \cross{0}{2}
    \cross{1}{2}
    \arcs{2}{2}

    \arcs{0}{3}
    \arcs{1}{3}

    \arcs{0}{4}
\end{tikzpicture}
\end{center}

Here, we can let $c_w$ be the leftmost crossing of row $2$ and $c_v$ be the rightmost crossing of row $3$. Then, $(c_w, c_v)$ form a bad parent-child pair: although $c_v$ is originally obstructed by the rightmost crossing $c_v'$ at the end of row $2$, sliding $c_v'$ out of the way with one simple ladder move leaves room for $c_v$ to slide into row $r_w$. Indeed, $w$ contains a $14523$ pattern. 

\end{example}

Using the notation $c_v, r_v, d_v$ as above, the following observations will be helpful:
\vspace{1em}

\begin{remark}
For a bad parent-child pair $(c_w, c_v)$, $d_v-d_w>1$ must be greater than $1$, otherwise $c_w$ itself will always obstruct $c_v$ from sliding into its row. 
\end{remark}
\vspace{1em}

\begin{remark}
For a bad parent-child pair $(c_w, c_v)$, $c_v$ will always be the rightmost crossing in the row $r_v$, since all remaining crossings will be its left children.
\end{remark}
\vspace{1em}

\begin{lemma}\label{lemma:sufficientcase1}
    Let $(c_w, c_v)$ be a bad parent-child pair, with corresponding rows $r_w, r_v$ and diagonals $d_w, d_v$. Then, if $c_v$ can slide up to $r_w$ via simple ladder moves without any other crossings performing any simple ladder moves, then $w$ contains one of the following patterns: $2413$, $2431$, $32154$, $341265$.
\end{lemma}
\begin{proof}
    Note $w(r_v)-w(r_w)>1$ (in order for $c_v$ to be able to slide up to $r_w$ unobstructed, we must have \[L(r_v) + (r_w-r_v) - L(r_w) \ge d_v-d_w>1 \implies w(r_v)-w(r_w)\ge L(r_v)+(r_w-r_v)-L(r_w)> 1\]
    where $L(r_v)$ denotes entry $r_v$ in the Lehmer code). 

    There is some row $x$ with $w(r_v)>w(x)>w(r_w)$. Note if $r_v>x>r_w$ then the rightmost crossing in row $x$ must be on the diagonal below the rightmost crossing in row $r_w$, and by Remark \ref{rmk:diagonalremark} $r_w$ would be in the subtree of $x$, which contradicts. Thus $x>r_v$. Since $w(r_w) \neq 1$, there is some $y$ with $w(y)<w(r_w)$. 
    
    \textbf{Subcase 1:} If $y>r_v$ also, then whether $x< y$ or $y < x$, $w$ contains either a $2413$ or a $2431$ pattern given by the indices $r_w, r_v, x,y$ or $r_w,r_v,y,x$.
    
    \textbf{Subcase 2:} Otherwise $r_w<y<r_v$ for all such $y$ (and if $L(r_w)=1$ then we must have $r_v-r_w=1$, so there must be at least $2$ such $y$). Additionally there must be some nonzero row between $r_w$ and $r_v$. Thus the possible cases are
    \begin{enumerate}
        \item $32154$
        \item $342165$
        \item $324165$
        \item $341265$
    \end{enumerate}
    The first three contain $32154$, and the fourth one is $341265$.
\end{proof}

\begin{lemma}\label{lemma:sufficientcase2}
    Suppose that $(c_w,c_v)$ is a bad parent-child pair. Then, suppose the rightmost crossing of row $r_v$ can slide up past $r_w$ only after some box in row $r_w$ makes a simple ladder move. (That is, there exists a crossing on diagonal $d_v$ or $d_{v}-1$). Then, $w$ contains either a $1432$, $2431$ or a $14523$ pattern.
\end{lemma}
\begin{proof}
    Suppose that $c_{v'}$ is our crossing on diagonal $d_v$ or $d_{v-1}$, between rows $r_v$ and $r_w$. First, suppose that $c_{v'}$ is not \textit{in} row $r_w$ (in $P_{\text{bott}}(w)$). But then, by Remark \ref{rmk:diagonalremark}, $c_w$ cannot directly be a parent of $c_v$, contradicting the assumptions of our lemma. (And inductively, we could keep searching for a bad parent-child pair by letting $c_{v'}$ play the role of $c_v$ instead).

    Therefore, we may assume that the only crossing on diagonal $d_v$ or $d_{v}-1$ between rows $r_v$ and $r_w$ is contained in row $r_w$ itself. In this case, notice that there must exist some row $r_1$ above $r_w$ satisfying \[w(r_1)<w(r_w),w(r_v)\] 
    (If not, then no crossing in $r_w$ can ever move starting from $P_{\text{bott}}(w)$, and $c_v$ can never slide into row $r_w$). This $r_1$ will play the role of the $1$ in our desired patterns. 
    
    Now, we break into two subcases:

    \textbf{Subcase 1:} There is a crossing in $r_w$ on diagonal $d_v$ (as well as $d_{v}-1$, since $P_{\text{bott}}(w)$ is left-adjusted). 

    In this case, $w$ contains a $1432$ pattern or a $2431$ pattern, formed by indices $r_1, r_w, r_v$, and some row $x > r_v$. More specifically, there must be an $x > r_v$ with $w(x) < w(r_v)$, so that $L(r_v) \neq 0$. Furthermore, if $w(r_1) < (x)$, we have a $1432$ pattern. If $r_1 > x$, we have a $2431$ pattern. 
    
    Finally, consider:

    \textbf{Subcase 2:} row $r_w$ contains a crossing in $d_v-1$, but not in $d_v$. 

    In this case, we claim that $w$ contains a $14523$ pattern or a $1432$ pattern. Because $c_v$ can slide along its diagonal until it is directly beneath the rightmost crossing of row $r_w$, we have $w(r_w) < w(r_v)$.

    Now, we claim that row $r_v$ must contain more than one crossing. If it only contains $c_v$, then Lemma \ref{definition:matching} implies that $c_v$ is actually the right child of the rightmost crossing in row $r_w$. But this is impossible, since we said the diagonals of $c_v$ and $c_w$ must be more than $1$ apart. 

    Therefore, $L(r_v) \geq 2$, so there are two more rows $x, y$ with $x, y > r_v$ and $w(r_v) > w(x), w(y)$. Finally, depending on the six possible relative orderings of $w(r_1), w(x),$ and $w(y)$, indices $r_1, r_w, r_v, x, y$ give us one of the following patterns in $w$: 
    \begin{enumerate}
        \item $15432$
        \item $25413$
        \item $25431$
        \item $35412$
        \item $35421$ or 
        \item $15423$
    \end{enumerate}

    In the first three cases, there is a $1432$ subpattern. In the next two cases, there is a $2431$ subpattern. And in the last case, we have our last forbidden pattern, $15423$. 
\end{proof}

Combining the two cases in Lemmas \ref{lemma:sufficientcase1} and \ref{lemma:sufficientcase2} gives us the sufficient direction of Theorem \ref{thm:patternavoidance}. 

\section{Acknowledgements}

We thank Lucas Gagnon for interesting and encouraging conversations about forest polynomials, and Linus Setiabrata for useful comments on the presentation of our proof. 

\bibliographystyle{plain}
\bibliography{ref}

@article{forestpolynomials1,
title = {Forest polynomials and the class of the permutahedral variety},
journal = {Advances in Mathematics},
volume = {453},
pages = {109834},
year = {2024},
issn = {0001-8708},
doi = {https://doi.org/10.1016/j.aim.2024.109834},
url = {https://www.sciencedirect.com/science/article/pii/S0001870824003499},
author = {Philippe Nadeau and Vasu Tewari},
}

@misc{quasisymmetricflag,
      title={The quasisymmetric flag variety: a toric complex on noncrossing partitions}, 
      author={ Bergeron, Nantel and  Gagnon, Lucas and  Nadeau, Philippe and  Spink, Hunter and  Tewari, Vasu},
      year={2025},
      eprint={2508.12171},
      archivePrefix={arXiv},
      primaryClass={math.AG},
      url={https://arxiv.org/abs/2508.12171}, 
}

@article{pipedreams,
  title   = {{R}{C}-graphs and {S}chubert Polynomials},
  author  = {Bergeron, Nantel and Billey, Sara},
  year    = {1993},
  journal = {Journal of Experimental Mathematics},
  volume  = {2},
  number  = {4},
  pages   = {257-269},
}

@article{principalspecializations,
  title   = {Principal Specializations of {S}chubert Polynomials and Pattern Containment},
  author  = {Gao, Yibo},
  year    = {2021},
  journal = {European Journal of Combinatorics},
  volume  = {94},
}

@misc{quasisymmetricdivideddifferences,
      title={Quasisymmetric divided differences}, 
      author={ Nadeau, Philippe and Spink, Hunter  and Tewari, Vasu },
      year={2024},
      eprint={2406.01510},
      archivePrefix={arXiv},
      primaryClass={math.CO},
      url={https://arxiv.org/abs/2406.01510}, 
}

@article{LascouxSchutzenberger1985,
  author  = {Lascoux, Alain and Sch{\"u}tzenberger, Marcel-Paul},
  title   = {Schubert polynomials},
  journal = {Topology},
  year    = {1985},
  volume  = {34},
  number  = {2},
  pages   = {229--254},
  doi     = {10.1016/0040-9383(85)90028-6}
}

@article{singleSEMs,
    author = {Woodruff, Dora},
title = {Single {S}{E}{M} {S}chubert Polynomials},
journal = {Electronic Journal of Combinatorics},
year = {2026}}

@article{01Schubert,
author = {Fink, Alex and Mészáros, Karola and St. Dizier, Avery}, 
title = {Zero-one {S}chubert polynomials}, 
journal = {Mathematische Zeitschrift}, 
year = {2020}, 
volume = {297},
pages = {1023--1042},
doi = {https://doi.org/10.1007/s00209-020-02544-2
}}

@article{132patterns,
     author = {Weigandt, Anna E.},
     title = {Schubert polynomials, 132-patterns, and {Stanley{\textquoteright}s} conjecture},
     journal = {Algebraic Combinatorics},
     pages = {415--423},
     year = {2018},
     publisher = {MathOA foundation},
     volume = {1},
     number = {4},
     doi = {10.5802/alco.27},
     mrnumber = {3875071},
     zbl = {1397.05205},
     language = {en},
     url = {https://www.numdam.org/articles/10.5802/alco.27/}
}

@article{pipedreams1,
     author = {Billey, Sara and Jockusch, William and Stanley, Richard},
     title = {Some Combinatorial Properties of {S}chubert Polynomials},
     journal = {Journal of Algebraic Combinatorics},
     pages = {345-374},
     year = {1993},
     publisher = {Kluwer Academic Publishers},
     volume = {2},
}

@article{pipedreams2,
     author = {Fomin, Sergey and Kirillov, Anatol},
     title = {The {Y}ang-{B}axter equation, symmetric functions, and {S}chubert polynomials},
     journal = {Discrete Mathematics},
     pages = {123-143},
     year = {1996},
     publisher = {Elsevier},
     volume = {153},
}

\end{document}